\documentclass[a4paper,10pt]{amsart}
\usepackage[utf8]{inputenc}
\usepackage{amsmath, amssymb,bm}
\usepackage[all]{xy}
\title{Albert Algebras over Rings and Related Torsors}
\author[S.\ Alsaody]{Seidon Alsaody}
\address{Department of Mathematical and Statistical Sciences, University of Alberta, Edmonton, AB T6G 2G1, Canada.}
\thanks{An extensive part of this work was done while the author was a postdoctoral fellow at Institut Camille Jordan (Lyon, France),
supported by the grant KAW 2015.0367 from the Knut and Alice Wallenberg Foundation.}

\newtheorem{Thm}{Theorem}[section]
\newtheorem{Prp}[Thm]{Proposition}
\newtheorem{Cor}[Thm]{Corollary}
\newtheorem{Lma}[Thm]{Lemma} 
\theoremstyle{definition}
\newtheorem{Def}[Thm]{Definition}
\newtheorem{Rk}[Thm]{Remark}
\newtheorem{Ex}[Thm]{Example}

\numberwithin{equation}{section}

\newcommand{\Hom}{\operatorname{Hom}}

\newcommand{\Ker}{\mathrm{Ker}}
\newcommand{\End}{\mathrm{End}}
\newcommand{\Aut}{\mathbf{Aut}}

\newcommand{\Z}{\mathbb{Z}}
\newcommand{\Gm}{\mathbb{G}_{\mathrm{m}}}

\newcommand{\Zor}{\mathrm{Zor}}

\newcommand{\Id}{\mathrm{Id}}

\newcommand{\Lie}{\mathrm{Lie}}

\newcommand{\bAut}{\mathbf{Aut}}

\newcommand{\bGL}{\mathbf{GL}}
\newcommand{\RT}{\mathbf{RT}}
\newcommand{\Spin}{\mathbf{Spin}}
\newcommand{\Spec}{\mathrm{Spec}}

\newcommand{\bH}{\mathbf{H}}
\newcommand{\bG}{\mathbf{G}}
\newcommand{\bF}{\mathbf{F}}

\newcommand{\bS}{\mathbf{S}}

\newcommand{\bE}{\mathbf{E}}
\newcommand{\bW}{\mathbf{W}}

\newcommand{\Isom}{\mathbf{Isom}}
\newcommand{\bStr}{\mathbf{Str}}
\newcommand{\bc}{\mathbf{c}}
\newcommand{\be}{\mathbf{e}}
\newcommand{\calM}{\mathcal{M}}
\newcommand{\barC}{\overline{C}}
\newcommand{\bone}{\mathbf{1}}

\entrymodifiers={+!!<0pt,\fontdimen22\textfont2>}
\begin{document}

\begin{abstract} We study exceptional Jordan algebras and related exceptional group schemes over commutative rings from a geometric point of view, using appropriate torsors to parametrize and explain
classical and new constructions, and proving that over rings, they give rise to non-isomorphic structures.

We begin by showing that isotopes
of Albert algebras are obtained as twists by a certain $\mathrm F_4$-torsor with total space a group of type $\mathrm E_6$, and using this, that Albert algebras over rings in general
admit non-isomorphic isotopes, even in the split case as opposed to the situation over fields. We then consider certain $\mathrm D_4$-torsors constructed from reduced Albert algebras,
and show how these give rise to a class of generalised reduced Albert algebras constructed from compositions of quadratic forms. Showing that this torsor is non-trivial, we conclude
that the Albert algebra does not uniquely determine the underlying composition, even in the split case. In a similar vein, we show that a given reduced Albert algebra can admit two coordinate
algebras which are non-isomorphic and have non-isometric quadratic forms, contrary, in a strong sense, to the case over fields, established by Albert and Jacobson.
\end{abstract}
\maketitle
\bigskip


\medskip


\date{\today}

\section{Introduction}
Albert algebras are exceptional Jordan algebras. Over fields, any simple Jordan algebra is either \emph{special}, i.e.\ a subalgebra of the anticommutator algebra $B^+$ 
of an associative algebra $B$, or an Albert algebra \cite{MZ}. Albert algebras over rings have attracted recent interest, as is seen in the extensive survey \cite{P}, which contains a historic account to which we refer the interested reader. 
Over an algebraically closed field, there is a unique Albert algebra up to isomorphism; this both 
testifies to the remarkable character of these algebras, and indicates that the framework of torsors, descent and cohomology is suitable for studying them over a general base. 
Furthermore, the assignment $A\mapsto \Aut(A)$ defines a correspondence between Albert algebras and groups of type $\mathrm F_4$. Thus understanding
Albert algebras provides a precise understanding of such groups, and vice versa. This is analogous to the relationship between octonion algebras and groups of type $\mathrm G_2$. 
In the present work, we  
approach Albert algebras using appropriate torsors (principal homogeneous spaces), building on this interplay between groups and algebras as well as the relation that Albert
algebras have to octonion algebras, triality for groups of type $\mathrm D_4$, and cubic norms. Thanks to this approach, we are able to shed new light on classical constructions, generalise them, and, expressing the objects as twists by appropriate torsors, we are able to obtain results that the classical methods
have fallen short of providing.

We begin by giving a sketch of the key ideas of the present paper,
with emphasis on the torsors
involved. In this introduction, we abuse notation by referring to groups simply by their type; this is done for the sake of overview and does not imply uniqueness of groups of a certain type.
Precise notation will be used in the respective sections.

An indication that Albert algebras over general unital, commutative rings can be expected to behave much less orderly than they do over fields is that this is true for octonion algebras. 
Indeed, Gille showed in \cite{G1}, by means of $\mathrm G_2$-torsors, that octonion
algebras over rings are not determined by their norm forms. In \cite{AG} the author, together with Gille, showed, using triality, how twisting an octonion algebra $C$ by 
such a torsor precisely gives the classical construction
of isotopes of $C$. Abusing notation, these papers dealt with the torsor\footnote{In this notation for torsors, the arrow denotes the structure map from the total space to the base, and
the group labelling the arrow is the structure group.}
\[\xymatrix@R-10pt{\mathrm D_4\ar[d]^{\mathrm G_2}\\ \mathbf{S}_C^2}\]
where $\mathrm D_4$ and $\mathrm G_2$ are (simply connected) groups of the corresponding types, and $\mathbf{S}_C$ is the octonionic unit sphere in $C$, which is an affine scheme.
Inspired by this, we set out, in the current work, to examine Albert algebras from this point of view.

We start in Section \ref{S2} with the question of isotopy, where we show that the isotopes of an Albert algebra $A$ are precisely the twists of $A$ by the torsor defined by the projection
of the isometry group of the cubic norm $N$ of $A$, which is simply connected of type $\mathrm E_6$, to the unit sphere of $N$. This is a torsor under the automorphism group of $A$; thus it is of the type
\[\xymatrix@R-10pt{\mathrm E_6\ar[d]^{\mathrm F_4}\\ \mathbf{S}_N}\]
with $\mathbf S_N$ an affine scheme. The non-triviality of this torsor for certain Albert algebras is therefore known, since it is
known that over e.g.\ $\mathbb R$ the Albert algebra $H_3(\mathbb O)$ of hermitian $3\times 3$-matrices over the real division octonion algebra $\mathbb O$ 
has non-isomorphic isotopes. However, over any field, isotopes of the split Albert algebra are isomorphic. We prove that this is not the case over rings in general.

In Section \ref{S3} we turn to the inclusion of groups of type $\mathrm D_4$ into the automorphism group of a \emph{reduced} Albert algebra, i.e.\ one of the form $H_3(C,\Gamma)$ for
an octonion algebra $C$ and a triple $\Gamma$ of invertible scalars. (This is a slight deformation of the algebra $H_3(C,\bone)=H_3(C)$ of hermitian matrices.) In \cite{ALM} it was shown that over fields,
this inclusion defines a $\mathrm D_4$-torsor where the base is the scheme of frames of idempotents of the Albert algebra, and this result easily generalises to rings. We show that over rings this torsor is non-trivial in general, even under the condition
that the Albert algebra be split. This torsor is of the type
\[\xymatrix@R-10pt{\mathrm F_4\ar[d]^{\mathrm D_4}\\ \mathbf{F}_A}\]
with $\mathbf F_A$ an affine scheme. In a sense, therefore, this torsor bridges the gap between the two torsors above. The natural question that arises then is what objects it parametrizes.
To answer it, we observe that the definition of a reduced Albert algebra becomes more clean if the symmetric algebras of para-octonions are used instead of octonion algebras. This is in
line with a recent philosophy followed in \cite{KMRT} and \cite{AG}. From there, we are able to generalise the construction of reduced Albert algebras by noting that one may replace the para-octonion algebra by
any \emph{composition of quadratic forms} of rank 8. Such objects generalise (para)-octonion algebras, since they consist of a triple $(C_1,C_2,C_3)$ of quadratic spaces, each of constant
rank 8 as a locally free module, with a map 
(the composition) $C_3\times C_2\to C_1$ that is compatible with the quadratic forms. If $C_1=C_2=C_3=C$, where $C$ is an octonion algebra, then the multiplication of $C$ is a composition
of quadratic forms, and every composition of rank 8 is locally isomorphic to a compositions arising from an octonion algebra in this way. We show that the $\mathrm D_4$-torsor in question classifies
those compositions of quadratic spaces that give rise to isomorphic Albert algebras via this generalised construction. Non-triviality of the torsor means that two non-isomorphic compositions may give 
rise to isomorphic Albert algebras. More precisely we show that even the algebra $H_3(C)$ can arise from compositions not isomorphic to the one obtained from $C$, even in the split case.

This gives rise to the yet more precise question of the extent to which a reduced Albert algebra $A=H_3(C)$ determines $C$, which we treat in Section \ref{S4}. Over a field $k$, the isomorphism class of $A$ (in fact even the 
isotopy class) determines $C$ up to isomorphism; this goes back to the 1957 paper \cite{AJ} of Albert and Jacobson. In cohomological terms, the map
\[H^1(k,\Aut(C))\to H^1(k,\Aut(A)),\]
induced by the natural inclusion $\Aut(C)\to\Aut(A)$, has trivial kernel. Since this inclusion factors through the simply connected group $\Spin(q_C)$ of the quadratic form $q_C$ of $C$,
we first show how the aforementioned result of \cite{G1}, augmented by results from \cite{AG}, readily implies that in general there exists an octonion algebra $C'\not\simeq C$ with $H_3(C')\simeq H_3(C)$,
namely by taking $C'$ having the same quadratic form as $C$. This corresponds to the torsor
\[\xymatrix@R-10pt{\mathrm F_4\ar[d]^{\mathrm G_2}\\ \mathrm F_4/\mathrm G_2}\]
being non-trivial. One may therefore ask if this is the only obstruction, i.e.\ if $H_3(C)$ at least determines the quadratic form of $C$ up to isometry. The main result of Section \ref{S4} is that this question
has a negative answer, which we show by proving that the torsor
\[\xymatrix@R-10pt{\mathrm F_4\overset{\mathrm G_2}{\wedge}\mathrm D_4\ar[d]^{\mathrm D_4}\\ \mathrm F_4/\mathrm G_2}\]
is non-trivial, where $\wedge$ denotes the contracted product.

A conclusion of our work is thus that the behaviour of Albert algebras over rings is significantly more intricate than it is over fields, even when the algebras are reduced or split, and that
these intricacies can be better understood using appropriate torsors. The combination of these two conclusions hopefully serves as an indication for future investigations.

\subsection{Preliminaries on Albert Algebras}
Throughout, $R$ is a unital, commutative ring. By an $R$-ring we mean a unital, commutative and associative $R$-algebra, and by an $R$-field we mean an $R$-ring that is a field. If $M$
is an $R$-module and $S$ is an $R$-ring, we write $M_S$ for the $S$-module $M\otimes_R S$ obtained by base change. All sheaves and cohomology sets involved will be with respect to the fppf topology, and we therefore omit the subscript fppf universally.

A \emph{Jordan algebra} over $R$ is an $R$-module $A$ endowed with a distinguished element $1$ and a quadratic map $U:A\to \End_R(A), x\mapsto U_x$, satisfying the identities
\[\begin{array}{llll}
U_1=\Id_A, & U_{U_xy}=U_xU_yU_x, & \text{and} & U_x\{yxz\}=\{xy(U_xz)\} 
  \end{array}
\]
for all $x,y,z\in A_S$ whenever $S$ is an $R$-ring, where $\{abc\}=(U_{a+c}-U_a-U_c)b$ is the \emph{triple product}. The map $U$ is called the $U$-operator of $A$ and replaces, in a sense,
the multiplication in linear Jordan algebras, and the element $1$ is the \emph{unity} of $A$.

\begin{Ex} The algebras we will consider will be \emph{cubic Jordan algebras}. Recall from \cite{P} that a \emph{cubic norm structure} over $R$ is an $R$-module $A$ endowed 
with a base point $1_A$, a cubic form $N=N_A:A\to R$, known as the \emph{norm},
and a quadratic map $A\to A$ denoted as $x\mapsto x^\sharp$ and known as the \emph{adjoint}. These are required to satisfy certain regularity and compatibility properties (see \cite[5.2]{P}) and give rise to the bilinear trace 
$T=T_A:A\times A\to R$, and the bilinear cross product $A\times A\to A$ given by $(x,y)\mapsto x\times y:=(x+y)^\sharp-x^\sharp-y^\sharp$. The cubic Jordan algebra associated to this
cubic norm structure is the module $A$ with unity $1=1_A$ and $U$-operator given by
\begin{equation}\label{U}U_xy=T(x,y)x-x^\sharp\times y.\end{equation}
\end{Ex}
An element $p\in A$ is invertible precisely when $N(p)\in R^*$, in which case $p^{-1}=N(p)^{-1}p^\sharp$. 

We refer to \cite{P} for a thorough discussion of cubic Jordan algebras. An \emph{Albert algebra} over $R$ is a cubic Jordan $R$-algebra $A$ the underlying module of which is projective of constant rank 27, and such that the Jordan
algebra $A\otimes_R k$ is simple for every $R$-field $k$. The notion of an Albert algebra is stable under base change.

\subsection{Preliminaries on Torsors}
We recall and slightly extend some known facts about torsors that will be useful later. Let $E$ be a $G$-torsor over $X$, where $E$ and $X$ are $R$-schemes, $X$ is affine, and 
$G$ is an $R$-group scheme. Recall that the topology in
question is always the fppf topology. In particular, we have a map $\Pi:E\to X$ and a compatible action of $G$ on $E$ (to fix a convention, we assume $G$ acts on the \emph{right}.)
Recall that $E$ is trivial if there is a $G$-equivariant morphism $E\to X\times G$ over $X$, which is equivalent to the projection $\Pi$ admitting a section \cite[III.4.1.5]{DG}.

\begin{Rk}\label{Rtrivial} If $E$ is trivial, then clearly $\Pi_S:E(S)\to X(S)$ is surjective for every $R$-ring $S$. We will need the converse, which holds as well: if $E$ is
non-trivial, then there is an $R$-ring $A$ such that $\Pi_A$ is not surjective. Indeed, if $X=\Spec(A)$ for an $R$-ring $A$, then $X(A)=\Hom(A,A)$ contains the generic element $\Id_A$. 
If $\Id_A=\Pi(f)$ for some $f\in E(A)$, then the morphism $X=\Spec(A)\to E$ corresponding to $f$ is a section. 
\end{Rk}

A certain class of torsors will be of particular interest to us.

\begin{Rk}\label{Rquotient} If $G$ is an $R$-group scheme, and $H$ is a subgroup scheme of $G$, 
then the inclusion $i:H\to G$ induces a map $i^*:H^1(R,H)\to H^1(R,G)$; explicitly, by \cite[III.4.4.1]{DG}, $i^*$ maps the class of an $H$-torsor $E$ to the class of the $G$-torsor
$E\wedge^H G$, where $\wedge^H$ denotes the contracted product over $H$. This is the quotient sheaf of $E\times G$ by the relation $\sim$ defined, on each $E(S)\times G(S)$,
by $(e,g)\sim(e',g')$ whenever $(eh,g)\sim(e',hg')$ for some $h\in H(S)$, and the structural projection of $E\wedge^H G$ is the projection on the first component. By \cite[III.4.4.5]{DG} (or \cite[2.4.3]{G2}), the kernel of $i^*$ is in bijection with the set of orbits
of the left action of $G(R)$ on $(G/H)(R)$, where $G/H$ is the fppf quotient. This bijection is given by assigning to the orbit of $x\in (G/H)(R)$ the class of the $H$-torsor
$\Pi^{-1}(x)$ where $\Pi$ is the quotient projection.
\end{Rk}

In order to produce examples of non-trivial torsors, we will use the following straight-forward generalisation of an argument from \cite{G1}. 

\begin{Lma}\label{Lhomotopy} Let $K\in\{\mathbb R,\mathbb C\}$, let $X$ be an affine $K$-scheme, $G$ an affine algebraic $K$-group, $E$ a $G$-torsor over $X$, and assume that $E$ has a $K$-point $e$. If for some $n\in\mathbb N$ the 
homotopy group $\pi_n(G(K),1_G)$ is not a direct factor of the homotopy group $\pi_n(E(K),e)$, then the torsor $E$ is non-trivial. 
\end{Lma}

The condition of the existence of a $K$-point does not imply that $E$ is trivial. Indeed this condition is fulfilled whenever $E$ is an affine algebraic $K$-group, $G$ a closed subgroup, and 
$X=E/G$, with $e$ being the neutral element of $E$.

\begin{proof} We mimic the proof of \cite{G1}. Assume that $E$ is trivial. Then there is a $G$-equivariant morphism $E\to X\times G$ over $X$. This implies the existence of a map
$\rho: E\to G$ that is a retraction of the inclusion $G\to E$ defined by $g\mapsto e_Sg$ for every $K$-ring $S$ and $g\in G(S)$. This in particular implies that the corresponding
inclusion $G(K)\to E(K)$ admits a continuous retraction. It follows that $\pi_n(G(K),1_G)$ is a direct factor of the homotopy group $\pi_n(E(K),e)$. The proof is complete. 
\end{proof}

\subsection{Acknowledgements} I am indebted to Philippe Gille for many fruitful discussions and valuable remarks. I am also grateful to Erhard Neher and
Holger Petersson for their interest and encouragement and for enriching conversations and comments.
                    
\section{Isometries, Isotopes and $\mathrm{F}_4$-torsors}\label{S2}
\subsection{Isometry and Automorphism Groups}
Let $A$ be an Albert algebra over $R$ with norm $N$. We begin by determining the groups $\Aut(A)$ and $\Isom(N)$. While this is done to varying degree in the existing literature, we
include the proofs of the below results as there is no single reference to which we can refer. Recall that the subgroup $\Isom(N)$ of $\bGL(A)$ is defined, for each $R$-ring $S$, by
\[\Isom(N)(S)=\{\phi\in \bGL(A)(S)|N_S\circ\phi=N_S\},\]
where $N_S$ is the norm of $A_S=A\otimes S$. Let further $\bS_N$ be the cubic sphere of $A$, i.e.\ the $R$-group functor defined, for each $R$-ring $S$, by
\[\bS_N(S)=\{x\in A_S|N_S(x)=1\}.\]
Clearly, $\Isom(N)$ acts on $\bS_N$ by $\phi\cdot x=\phi(x)$ for each $R$-ring $S$, $x\in\bS_N(S)$ and $\phi\in\Isom(N)(S)$. We shall refer to this simply as \emph{the action of $\Isom(N)$ on $\bS_N$}.

\begin{Prp} The group $\Aut(A)$ is a semisimple algebraic group of type $\mathrm F_4$, and $\Isom(N)$ is a semisimple
simply connected algebraic group of type $\mathrm E_6$.
\end{Prp}

In the case where $R$ is a field of characteristic different from 2 and 3, the result is due to \cite[Theorems 7.2.1 and 7.3.2]{SV}

It is known (see \cite[Theorem 17]{P}) that there is a faithfully flat $R$-ring $S$ such that
$A_S$ is split, hence isomorphic to $A^s\otimes_\Z S$, where $A^s$ is the split $\Z$-Albert algebra $H_3(C^s)$ with norm $N^s$, $C^s$ being the split octonion algebra over $\Z$.
To unburden notation, we set $\bG=\Isom(N)$ and $\bH=\bAut(A)$, which are group schemes over $R$, and $\bG^s=\Isom(N^s)$ and $\bH^s=\bAut(A^s)$, which are group schemes over $\Z$.

\begin{proof} Since $\bG$ and $\bH$ are forms of $\bG^s_R$ and $\bH^s_R$, respectively, for the fppf topology, it suffices to establish the claims for $\bG^s$ and $\bH^s$, i.e.\ to 
prove that they are smooth affine $\Z$-group schemes with 
connected, semisimple, simply connected geometric fibres of the appropriate types. Now $\bG^s$ and $\bH^s$ are affine since they are closed subschemes of the affine scheme $\bGL(A^s)$. To prove smoothness, it suffices, by \cite[Lemma B.1]{AG}, to show that $\bG^s$ and $\bH^s$
are finitely presented and fibre-wise (i.e.\ over any field) smooth, connected and equidimensional. Finite presentation is clear. 
Since a group over a field is smooth if and only if its scalar extension to
an algebraic closure is, smoothness of the fibres follows from the lemma below, which also implies that the fibres are connected and equidimensional. 
Thus the groups are smooth, and the lemma also implies that they have semisimple, simply connected geometric fibres of the appropriate types. The proof is then complete.
\end{proof}

\begin{Lma}\label{L1} Assume that $R=k$ is an algebraically closed field. Then $\Aut(A)$ is a semisimple algebraic group of type $\mathrm F_4$, and $\Isom(N)$ is a semisimple
simply connected algebraic group of type $\mathrm E_6$. 
\end{Lma}

\begin{proof} First we address smoothness. The group $\Aut(A)$ is smooth (see e.g.\ \cite{ALM}). For $\Isom(N)$, consider the exact sequence of 
$k$-group schemes
\[\xymatrix{1\ar[r]&\Isom(N)\ar[r]^f&\bStr(A)\ar[r]^\mu&\Gm\ar[r]&1},\]
where the structure group $\bStr(A)$ is the group of norm similarities with respect to $N$ (the field being infinite) and $\mu$ is the multiplier map $\phi\mapsto\mu_\phi=N(\phi(1))$.
By \cite[Corollary 6.6]{L}, $\bStr(A)$ is smooth. For $\Isom(N)$ to be smooth it is necessary and sufficient that the differential $d\mu$ be a surjective map of Lie algebras.
But $\Id+\varepsilon\Id\in\bStr(A)(k[\varepsilon])$ (where $\varepsilon^2=0$) and $\mu_{\Id+\varepsilon\Id}=1+\varepsilon$. Thus $\Id\in\Lie(\bStr(A))$, and its image under $d\mu$ is $1\in k=\Lie(\Gm)$. Surjectivity
follows by linearity. 

Let $j$ be the inversion on $A$, i.e.\ the birational map on $A$ defined by $x\mapsto x^{-1}$ for all $x\in A^*$. Then $\Sigma=(A,j,1_A)$ is an H-structure in the sense of 
\cite{M}, as well as a J-structure in the sense of \cite{S}. By \cite{S}, the structure group $\bS(j)$ of $j$ is a smooth closed subgroup of $\bGL(A)$, and so is the 
automorphism group of $\Sigma$ as a J-structure. By definition a linear automorphism of $A$ is an automorphism of $\Sigma$ as a J-structure if and only if it is an automorphism of $\Sigma$ as an 
H-structure. By \cite{M} we thus have $\bAut(\Sigma)(k)\simeq\bAut(A)(k)$ and $\bS(j)(k)\simeq\bStr(A)(k)$. Thus since $k$ is algebraically closed, and all groups involved are smooth, 
$\bAut(A)\simeq\bAut(\Sigma)$ and $\bStr(A)\simeq\bS(j)$.
The statement about $\bAut(A)$ then follows from \cite[14.20]{S}, as does the statement that $\bStr(A)$ is the product of its one-dimensional centre and a semisimple, simply 
connected algebraic group $\bG'$ of type $\mathrm E_6$. From this and the above exact sequence, it follows that $f$ induces an isomorphism $\Isom(N)^\circ\to \bG'$, and it remains to be
shown that $\Isom(N)$ is connected. By Proposition 3.7 and Lemma 3.9 of \cite{ALM}, the group $\bStr(A)(k)$ acts transitively on $A^*$ ($k$ being algebraically closed). If $a\in A^*$ has norm 1
and $\phi\in\bStr(A)(k)$ satisfies $\phi(1)=a$, then necessarily $\phi\in\Isom(N)(k)$. Thus $\Isom(N)$ acts transitively on $\bS_N$, which is an irreducible $k$-variety since $N$ is irreducible.
(The irreducibility of $N$ follows from the fact that its restriction on $M_3(k)$ is the usual determinant, which is irreducible, as shown in e.g.\ \cite[Theorem 7.2]{J2}.) By \cite{ALM}, the stabiliser of $1$ is $\bAut(A)$, which
is connected. Thus $\Isom(N)$ is connected. Hence it is a semisimple, simply connected group of type $\mathrm E_6$, and the proof is complete.
\end{proof}

\subsection{Isotopes}
If $p\in A^*$, then the \emph{$p$-isotope} $A^{(p)}$ of $A$ is the Jordan algebra with unity $p^{-1}$ and $U$-operator $x\mapsto U^{(p)}_x:=U_xU_p$, where $U$ is the
$U$-operator of $A$. It is known that $A^{(p)}$ is an Albert algebra, and one can check that the map $\lambda U_p:A^{(\lambda U_pp)}\to A^{(p)}$ is an isomorphism
for any $\lambda\in R^*$, in particular for $\lambda=N(p)^{-1}$, in which case $N(\lambda U_pp)=1$. Thus up to isomorphism one may assume that $N(p)=1$. In that case the norm of $A^{(p)}$ coincides
with that of $A$, and the adjoint is given by $x\mapsto U_{p^{-1}}x^\sharp$. In this section we will show that isotopes of Albert algebras are twists by a certain torsor. To begin with, the following
proposition provides the torsor.

\begin{Prp}\label{Psphere} The stabiliser of $1_A$ with respect to the action of $\Isom(N)$ on $\bS_N$ is $\bAut(A)$, and the fppf quotient sheaf $\Isom(N)/\bAut(A)$ is representable by a smooth scheme. Furthermore, the map $\Pi:\Isom(N)\to\bS_N$, defined by $\phi\mapsto \phi(1)^{-1}$ for any $R$-ring $S$ and $\phi\in\Isom(N)(S)$, induces an isomorphism between the quotient
and $\bS_N$. 
\end{Prp}

\begin{proof} By the results of \cite{ALM} quoted in the proof of Lemma \ref{L1}, the action is transitive (in the sense of being transitive on geometric fibres) 
and the stabiliser of $1$ is $\bAut(A)$. Then since $\bAut(A)$ is flat (in fact smooth) and $\Isom(N)$ is smooth, $\Isom(N)/\bAut(A)$ is representable by a smooth scheme by 
\cite[XVI.2.2 and VIB.9.2]{SGA3}. The induced map
is then an isomorphism by \cite[III.3.2.1]{DG}.
\end{proof}

This defines an $\bAut(A)$-torsor over $\bS_N$ for the fppf topology, 
and we denote by $\bE^p$ the fibre of $p\in \bS_N(R)$, which is an $\bAut(A)$-torsor over $\Spec(R)$. The next theorem shows that the isotopes of $A$ are precisely the twists of $A$ 
by this torsor. Before formulating it, we shall
define precisely what we mean by a twist of a (quadratic) Jordan algebra.

Let $A$ be a Jordan algebra and $\bE$ a (right) $\bAut(A)$-torsor over $\Spec(R)$, and denote by $\bW(A)$ the vector group scheme defined by 
$\bW(A)(S)=A_S$ for
every $R$-ring $S$. The \emph{twist $\bE\wedge A$ of $A$ by $\bE$} is the sheaf of Jordan algebras associated to the presheaf
\[S\mapsto (\bE(S)\times A_S)/\sim\]
defined as follows: for each $R$-ring $S$, the equivalence relation $\sim$ is defined by $(\phi,a)\sim(\psi,b)$ if and only if $(\phi \rho^{-1},\rho(a))=(\psi,b)$ for some 
$\rho\in\bAut(A)(S)$. Denoting the equivalence classes by square brackets, the additive structure is given by
\[[\phi,x]+[\psi,y]=[\phi,x+\phi^{-1}\psi(y)],\]
the $S$-module structure by
\[\lambda[\phi,x]=[\phi,\lambda x],\]
the unity by $[\phi,1_A]$, and the $U$-operator by
\[U_{[\phi,x]}[\psi,y]=[\phi,U_x\phi^{-1}\psi(y)],\]
for any $\phi,\psi\in \bE(S)$, $x,y\in A_S$ and $\lambda\in S$. It is straight-forward to check that these expressions are well-defined on equivalence classes, and satisfy the axioms
of a Jordan algebra. 

\begin{Rk} If $\phi,\psi\in \bE(S)$, then $\rho=\phi^{-1}\psi\in\bAut(A)(S)$. Thus
\[[\psi,y]=[\psi\rho^{-1},\rho(y)]=[\phi,\phi^{-1}\psi(y)].\]
Thus the above definition is equivalent to the definition
\[U_{[\phi,x]}[\phi,y]=[\phi,U_xy].\]
Note further that in an algebra with bilinear multiplication, the twist of the multiplication is defined by
\[(\phi,x)(\psi,y)=(\phi,x\phi^{-1}\psi(y)).\]
If $2\in R^*$, then $A$ is endowed with a bilinear multiplication, and the $U$-operator is defied by
\[U_xy=2x(xy)-x^2y,\]
and one easily checks that the twist of the $U$-operator defined above is precisely the one obtained in this way from the twist of the bilinear multiplication.
\end{Rk}

\begin{Thm} Let $p\in \bS_N(R)$. The map
\[\Theta_S:[\phi,x]\to \phi(x)\]
for each $R$-ring $S$, $\phi\in\bE^p(S)$, and $x\in A_S$, defines a natural isomorphism of algebras $\bE^p\wedge A\to A^{(p)}$. 
\end{Thm}

The reader may wish to compare this result to \cite[Theorem 4.6]{AG}.

\begin{proof} The proof that this is a well-defined, injective linear map from $(\bE^p(S)\times A_S)/\sim$ to $A_S^{(p)}$ is straight-forward and analogous to that of \cite[Theorem 4.6]{AG}.
The same goes for surjectivity whenever $\bE^p(S)\neq\emptyset$, and for naturality. This proves that the map in question is a an isomorphism of sheaves of modules. It is a morphism of algebras
since $\Theta_S([\phi,1])=\phi(1)=p^{-1}=1_{A^{(p)}}$, and
\[\Theta_S(U_{[\phi,x]}[\psi,y])=\Theta_S([\phi,U_x\phi^{-1}\psi(y)])=\phi(U_x\phi^{-1}\psi(y)).\]
Now $\phi\in\bE^p(S)$ implies that $\phi\in\Isom(N)(S)$ with $\phi(1_A)=p^{-1}=1_{A^{(p)}}$. Thus $\phi:A\overset{\sim}{\rightarrow} A^{(p)}$, and
\[U_{\phi(x)}^{(p)}\phi(\phi^{-1}\psi(y))=U_{\phi(x)}^{(p)}\psi(y)=U_{\Theta_S[\phi,x]}^{(p)}\Theta_S[\psi,y],\]
which completes the proof. 
\end{proof}

\subsection{Non-Trivial Isotopes}
It is well-known that the above torsor is non-trivial in general; indeed, over $\mathbb R$, the Albert algebras form two isotopy classes, one of which consists of two isomorphism
classes (see \cite[13.7]{P}). However, by e.g.\ \cite[Proposition 20]{P}, any isotope of the split Albert algebra over a field $K$ is isomorphic it. As the next result shows, this is
not the case over rings.

\begin{Thm}\label{TF4} There exists a smooth $\mathbb C$-ring $S$ such that the split Albert $S$-algebra admits a non-isomorphic isotope.
\end{Thm}

\begin{proof} Let $A$ be the (split) complex Albert algebra, and denote its cubic norm by $N$. 
By Proposition \ref{Psphere}, the affine scheme $\bS_N$ is smooth. In view of this it suffices, by Remark \ref{Rtrivial}, to show that the $\Aut(A)$-torsor $\Isom(N)\to\bS_N$ of Proposition \ref{Psphere} is non-trivial. 
To this end it suffices, by Lemma \ref{Lhomotopy}, to show that for some $n$,
$\pi_n(\Aut(A)(\mathbb C))$ is not a direct factor of $\pi_n(\Isom(N)(\mathbb C))$. (The choice of base point is immaterial by connectedness.) 

Let $C$ be the real division octonion algebra, and set $A_0=H_3(C,\bone)$, with norm $N_0$, noting that $A=A_0\otimes_{\mathbb R}\mathbb C$.
Then $\Aut(A_0)$ (resp.\ $\Isom(N_0)$) is an anisotropic real group of type $\mathrm F_4$ (resp.\ $\mathrm E_6$). 
The Lie group $\Aut(A)(\mathbb C)=\Aut(A_0)(\mathbb C)$ (resp.\ $\Isom(N)(\mathbb C)=\Isom(N_0)(\mathbb C)$) contains $\Aut(A_0)(\mathbb R)$ (resp.\ $\Isom(N_0)(\mathbb R)$) as a maximal compact subgroup.
By Cartan decomposition, $\Aut(A)(\mathbb C)$ is thus homeomorphic to $\Aut(A_0)(\mathbb R)\times\mathbb R^k$ and $\Isom(N)(\mathbb C)$ is homeomorphic to $\Isom(N_0)(\mathbb R)
\times\mathbb R^l$ for suitable $k$ and $l$. By \cite{Mi}, $\pi_8(\Aut(A_0)(\mathbb R))$, and hence also $\pi_8(\Aut(A)(\mathbb C))$, 
is cyclic or order 2, while by \cite{BS}, $\pi_8(\Isom(N_0)(\mathbb R))$, and thus also $\pi_8(\Isom(N^s)(\mathbb C))$, is trivial. This completes the proof.
\end{proof}

\section{Compositions of Quadratic Forms and $\mathrm{D}_4$-torsors}\label{S3}
In this section, we will define and consider Albert algebras arising from compositions of quadratic forms, and describe these in terms of certain $\mathrm{D}_4$-torsors.

\subsection{Composition Algebras and Compositions of Quadratic Forms}
An Albert algebra over $R$ is called \emph{reduced} if it is isomorphic to the algebra $H_3(C,\Gamma)$ of twisted hermitian matrices over an octonion $R$-algebra $C$, with
$\Gamma=(\gamma_1,\gamma_2,\gamma_3)\in (R^*)^3$. Before giving the definition of $H_3(C,\Gamma)$, we note that the formulae can be simplified slightly if they are expressed in terms of the \emph{para-octonion algebra obtained from $C$}. This is the algebra with underlying
quadratic module $C$, and multiplication given by
\[x\cdot y=\bar x\ \bar y,\]
where $a\mapsto \bar a$ denotes the usual involution on $C$ and juxtaposition the multiplication of $C$. This algebra is not unital, but it is 
a symmetric composition algebra in the sense of the following definition.

\begin{Def} A \emph{composition algebra} $(C,q)$ over $R$ is a faithfully projective $R$-module $C$ endowed with a bilinear map $C\times C\to C, (x,y)\mapsto x\cdot y$ (the multiplication) and a 
nonsingular quadratic form $q=q_C$ satisfying $q(x\cdot y)=q(x)q(y)$ for all $x,y\in C$. The algebra is called \emph{symmetric} if the symmetric bilinear form
\[(x,y)\mapsto \langle x,y\rangle=q(x+y)-q(x)-q(y)\]
satisfies 
\[\langle x\cdot y, z\rangle=\langle x,y\cdot z\rangle\]
for all $x,y,z\in C$. 
\end{Def}

We shall only be interested in symmetric composition algebras of \emph{constant rank 8}. Examples of these include para-octonion algebras and (over fields) Okubo algebras. In this case, the module being faithfully projective amounts to it being projective (of constant rank 8), and
non-singularity of $q$ is equivalent to non-degeneracy of $\langle,\rangle$. Abusing notation, we shall refer to a composition algebra $(C,q)$ simply by $C$.

We shall begin by writing the definition of reduced Albert algebras in terms of symmetric composition algebras. Let thus $C$ be a symmetric composition algebra over $R$ of 
constant rank 8. We will often consider the $R$-module $R^3\times C^3$. Following the notation of \cite{P}, we write an 
element $(\alpha_1,\alpha_2,\alpha_3,u_1,u_2,u_3)\in R^3\times C^3$ as $\sum \alpha_i e_i +\sum u_i[jl]$ where the second sum runs over all cyclic permutations $(i,j,l)$ of
$(1,2,3)$. We write $\Delta$ for the trilinear form on $C$ defined by
\[\Delta(u_1,u_2,u_3)=\langle u_3\cdot u_2,u_1\rangle,\]
noting that by symmetry of $C$, it is invariant under cyclic permutations.

\begin{Prp}\label{PComp} Let $C$ be a symmetric composition algebra over $R$ of constant rank 8 and let $\Gamma=(\gamma_1,\gamma_2,\gamma_3)\in R^3$ with $\gamma_1\gamma_2\gamma_3\in R^*$. 
Then $R^3\times C^3$ is a cubic norm structure with base point $e=\sum e_i$ and norm $N$ and adjoint map $\sharp$ defined by
\[N(x)=\alpha_1\alpha_2\alpha_3+\gamma_1\gamma_2\gamma_3\Delta(u_1,u_2,u_3)-\sum\gamma
a_j\gamma_l\alpha_i q(u_i)\]
and
\[x^\sharp=\sum(\alpha_j\alpha_l-\gamma_j\gamma_lq(u_i))e_i + \sum (\gamma_iu_l\cdot u_j-\alpha_iu_i)[jl],\]
where $x=\sum \alpha_i e_i +\sum u_i[jl]$.
Moreover, the bilinear trace $T$ of this cubic norm structure satisfies
\[T(x,y)=\sum \alpha_i\beta_i+\sum \gamma_j\gamma_l\langle u_i,v_i\rangle,\]
where $y=\sum \beta_i e_i +\sum v_i[jl]$.
\end{Prp}

\begin{proof} It is straight-forward to verify that $N$, $\sharp$ and $1$ satisfy the axioms of \cite{P}, and that the bilinear trace defined in \cite{P} is equal to $T$.
 
\end{proof}

We denote the Jordan algebra of the above cubic form by $H(C,\Gamma)$.

\begin{Rk} If $C$ is a para-octonion algebra obtained from an octonion algebra $O$, then $H(C,\Gamma)=H_3(O,\Gamma)$, and the above definition coincides with the definition
of $H_3(O,\Gamma)$ in \cite{P}. 
\end{Rk}

For a para-octonion algebra $C$, we write $\RT(C)$ for the affine group scheme defined by
\[\RT(C)(S)=\{(t_1,t_2,t_3)\in\mathbf{SO}(q_C)(S)^3|t_1(x\cdot y)=t_3(x)\cdot t_2(y)\}.\]
Up to renaming the indices, by \cite{AG}, this is precisely the semi-simple simply connected group of type $\mathrm D_4$ denoted in \cite{AG} by $\RT(O)$, where $O$ is the 
octonion algebra from which $C$ is obtained.

For later use, we recall the definition of a composition of quadratic forms, generalising that of a composition algebra.

\begin{Def} A \emph{composition of quadratic forms over $R$} is a heptuple
\[\calM=(C_1,C_2,C_3,q_1,q_2,q_3,m),\] 
where $C_1$, $C_2$ and $C_3$ are projective 
$R$-modules of the same rank, and each $q_i$ is a non-singular quadratic form on $C_i$, and where 
$m:C_3\times C_2\to C_1$ is a bilinear map satisfying $q_1(m(x,y))=q_3(x)q_2(y)$ for any $x\in C_3$ and $y\in C_2$. We call the rank of $C_i$ the \emph{rank} of $\calM$. If 
$\calM'=(C'_1,C'_2,C'_3,q'_1,q'_2,q'_3,m')$ is another composition, then a \emph{morphism of compositions of quadratic forms} is a triple $(t_1,t_2,t_3)$ of isometries $t_i:C_i\to C'_i$
such that $m'(t_3(x),t_2(y))=t_1m(x,y)$ for each $x\in C_3$ and $y\in C_2$.
\end{Def}

It is clear that this notion is stable under base change. The notions of isomorphisms and automorphisms are clear, and we thus obtain an affine group scheme $\Aut(\calM)$.

\begin{Ex} Let $C$ be a composition algebra over $R$, with quadratic form $q_C$. 
The multiplication map $m_C:C\times C\to C$ makes $\calM_C:=(C,C,C,q_C,q_C,q_C,m_C)$ into a composition of quadratic forms. If $C$ is a para-octonion algebra, then by definition, $\Aut(\calM_C)=\RT(C)$.
\end{Ex}

Given a composition of quadratic forms $\calM=(C_1,C_2,C_3,q_1,q_2,q_3,m)$ of constant rank 8, the regularity of $q_2$ and $q_3$ implies that the map $m$ determines maps $m_2:C_1\times C_3\to C_2$ and $m_3:C_2\times C_1\to C_3$ by
\begin{equation}\label{companions}
 \begin{array}{l}
  \langle m_2(x_1,x_3),x_2\rangle_2=\langle x_1, m(x_3,x_2)\rangle_1,\\ \langle x_3, m_3(x_2,x_1)\rangle_3=\langle m(x_3,x_2),x_1\rangle_1
  \end{array}
\end{equation}
where $\langle,\rangle_i$ is the bilinear form corresponding to the quadratic form $q_i$. The significance of these is due to the following proposition.

\begin{Prp}\label{PExt} Let $\calM=(C_1,C_2,C_3,q_1,q_2,q_3,m)$ be a composition of quadratic forms over $R$ of constant rank 8. 
Then $\calM\otimes_R S\simeq_S\calM_C\otimes_R S$ for some para-octonion $R$-algebra $C$ and some faithfully flat $R$-ring $S$. Under any isomorphism $\calM\otimes_R S\to\calM_C\otimes_R S$, 
the multiplication on $C_S$ restricts to the 
maps $m_i:C_l\times C_j\to C_i$ for every
cyclic permutation $(i,j,l)$ of $(1,2,3)$, where $m_1=m$ and $m_2$ and $m_3$ are as in \eqref{companions}.
\end{Prp}

To simplify notation we shall write $\calM\otimes_R S$ as $(C_{1S},C_{2S},C_{3S},q_{1S},q_{2S},q_{3S},m_S)$.

\begin{Rk} It suffices to prove that $\calM\otimes_R S\simeq \calM_C$ for some para-octonion algebra $C$ \emph{over $S$}. Indeed, assume this is the case. If $O$
is the octonion
$S$-algebra from which $C$ is obtained, then $O\otimes_S T\simeq \Zor(R)\otimes_R T$ for some faithfully flat $S$-ring $T$, where $\Zor(R)$ is the split (Zorn) octonion $R$-algebra. Hence
$C\otimes_S T\simeq Z\otimes_R T$, where $Z$ is the para-octonion $R$-algebra obtained from $\Zor(R)$, and \emph{a fortiori} 
$\calM_C\otimes_S T\simeq \calM_{Z\otimes S}\otimes_S T\simeq \calM_Z\otimes_R T $. Altogether,
\[\calM\otimes_R T\simeq (\calM\otimes_R S)\otimes_S T\simeq\calM_C\otimes_S T\simeq\calM_Z\otimes_R T,\]
and the claim follows since $T$ is faithfully flat over $R$.
\end{Rk}

\begin{proof}[Proof of Proposition \ref{PExt}] Choose $S$ such that there are $a\in C_{3S}$ and $b\in C_{2S}$ with $q_{3S}(a)=q_{2S}(b)=1$. (This is possible since a regular quadratic form is split by a faithfully
flat extension.) Define
\[\begin{array}{ll}
f:C_{2S}\to C_{1S},& x\mapsto m_S(a,x),\\
g:C_{3S}\to C_{1S},& x\mapsto m_S(x,b).   
  \end{array}
\]
Then $q_{1S}f(x)=q_{2S}(x)$ and $q_{1S}g(x)=q_{3S}(x)$, whence $f$ and $g$ are isometries, hence invertible. Thus
\[(x,y)\mapsto x\cdot y:=m_S(g^{-1}(x),f^{-1}(y))\]
defines a bilinear multiplication on $C_{1S}$. Since $f^{-1}$ and $g^{-1}$ are isometries, we have
\[q_{1S}(x\cdot y)=q_{3S}g^{-1}(x)q_{2S}f^{-1}(y)=q_{1S}(x)q_{1S}(y).\]
Thus $q_{1S}$ is multiplicative with respect to the multiplication $\cdot$. Setting 
$e=m_S(a,b)$ we have
$e=f(b)=g(a)$. Thus $e\cdot x=ff^{-1}(x)=x$, and likewise $x\cdot e=x$, for all $x\in C_{1S}$. Thus the multiplication is unital, whence it makes
$C_{1S}$ into an octonion algebra over $S$ with quadratic form $q_{1S}$, and
\[(\Id,f,g):\calM\otimes S\to \calM_{C_{1S}}\]
is an isomorphism of compositions of quadratic forms since $g(x)\cdot f(y)=m_S(x,y)$ by construction. Denoting by $\kappa$ the involution on the octonion algebra
$C_{1S}$ and by $C$ the corresponding
para-octonion algebra, we have an isomorphism
\[(\Id,\kappa f,\kappa g): \calM\otimes S\to\calM_C,\]
and the first statement is proved.

To prove the second statement for $i=1$, let $x\in C_3$ and $y\in C_2$. Since any isomorphism $\calM\otimes_R S\to\calM_C\otimes_R S$ takes $m_S$ to the multiplication of $C_S$, it
suffices to show that $z:=m_S(x\otimes 1,y\otimes 1)\in C_{1S}$ descends to $C_1$, which is equivalent to
$z\otimes 1$ being fixed by the standard descend datum 
\[\begin{array}{ll}\theta:C_1\otimes S\otimes S\to C_1\otimes S\otimes S,& c\otimes \mu\otimes\nu\mapsto c\otimes \nu\otimes\mu\end{array}.\]
This in turn follows from the fact that
\[m_S(x\otimes 1,y\otimes 1)=m(x,y)\otimes 1\otimes 1.\]
To prove the statement for $i=3$, we let instead $x\in C_2$ and $y\in C_1$. We must show that for any symmetric composition algebra $C$ with multiplication $\bullet$ over $S$ and any
isomorphism of compositions of quadratic forms
\[(\phi_1,\phi_2,\phi_3): \calM\otimes S\to\calM_C\]
we have
\[\phi_3^{-1}(\phi_2(x\otimes 1)\bullet\phi_1(y\otimes 1))=m_3(x,y)\otimes 1.\]
Since $q_{3S}$ is non-degenerate, it suffices for this to show that for any $z\in C_3$,
\[\langle\phi_3^{-1}(\phi_2(x\otimes 1)\bullet\phi_1(y\otimes 1)),z\otimes 1\rangle_{q_{3S}}=\langle m_3(x,y)\otimes 1,z\otimes 1\rangle_{q_{3S}}.\]
Since $\phi_3:C_{3S}\to C$ is an isometry, the left hand side equals
\[\langle\phi_2(x\otimes 1)\bullet\phi_1(y\otimes 1),\phi_3(z\otimes 1)\rangle_{q_C}\]
which, by symmetry, is equal to
\[\langle\phi_3(z\otimes 1)\bullet\phi_2(x\otimes 1),\phi_1(y\otimes 1)\rangle_{q_C}=\langle \phi_1m_S(z\otimes 1,x\otimes 1),\phi_1(y\otimes 1)\rangle_{q_C}.\]
Since $\phi_1$ is an isometry, this is equal to
\[\langle m_S(z\otimes 1,x\otimes 1),y\otimes 1\rangle_{q_{1S}}=\langle m(z,x),y\rangle_1\otimes 1,\]
which by definition of $m_3$ is equal to
\[\langle m_3(x,y),z\rangle_3\otimes 1=\langle m_3(x,y)\otimes 1,z\otimes 1\rangle_{q_{3S}}\]
as desired. The statement for $i=2$ is deduced analogously, and the proof is complete.
\end{proof}

\subsection{Albert Algebras Constructed from Compositions}\label{Storsor}

Given a composition of quadratic forms
\[\calM=(C_1,C_2,C_3,q_1,q_2,q_3,m)\]
of constant rank 8, consider the $R$-module $R^3\times C_1\times C_2\times C_3$, whose elements we will write, in analogy to the above, as $x=\sum \alpha_i e_i +\sum u_i[jl]$, where $\alpha_i\in R$
and $u_i\in C_i$, and where the second sum runs over all cyclic permutations $(i,j,l)$ of $(1,2,3)$.

\begin{Prp}\label{PCOQF} The module $R^3\times C_1\times C_2\times C_3$ is a cubic norm structure with base point $e=\sum e_i$ and norm $N$ and adjoint map $\sharp$ defined by
\[N(x)=\alpha_1\alpha_2\alpha_3+\gamma_1\gamma_2\gamma_3\langle m(u_3,u_2),u_1\rangle_1-\sum\gamma a_j\gamma_l\alpha_i q_i(u_i)\]
and
\[x^\sharp=\sum(\alpha_j\alpha_l-\gamma_j\gamma_lq_i(u_i))e_i + \sum (\gamma_i m_i(u_l,u_j)-\alpha_iu_i)[jl].\]
Moreover, the bilinear trace $T$ of this cubic norm structure satisfies
\[T(x,y)=\sum \alpha_i\beta_i+\sum \gamma_j\gamma_l\langle u_i,v_i\rangle_i,\]
where $y=\sum \beta_i e_i +\sum v_i[jl]$. The cubic Jordan algebra of this cubic norm structure is an Albert algebra.
\end{Prp}

We denote this cubic Jordan algebra by $H(\calM,\Gamma)$. If $\calM=\calM_C$ for some symmetric composition algebra $C$, then $H(\calM,\Gamma)=H(C,\Gamma)$.

\begin{proof} We need to verify the identities (1) of \cite[5.2]{P} in all scalar extensions of $R$. It suffices to do this upon replacing $R$ by a faithfully flat $R$-ring $S$.
By Proposition \ref{PExt}, we can choose $S$ so that $\calM\otimes S\simeq \calM_C$ for some symmetric composition $S$-algebra $C$. Extending any isomorphism 
$\calM\otimes S\simeq \calM_C$ by the
identity on $S^3$, we obtain a linear bijection $H(\calM,\Gamma)\otimes S=H(\calM\otimes S,\Gamma)\to H(C,\Gamma)$ that transforms the expressions
for the norm, adjoint and trace into those of Proposition \ref{PComp}, from which the first two statements therefore follow. We conclude with \cite[Theorem 17]{P} and Proposition \ref{PComp}. 
\end{proof}

\begin{Rk} A related, but different, construction over fields is the Springer construction, which provides an Albert algebra given a \emph{twisted composition}, see \cite[38.6]{KMRT}. Twisted compositions
over a field $k$ involve a quadratic space over a cubic \'etale algebra over $k$; eight-dimensional (symmetric) composition algebras over $k$ give rise to twisted compositions using the split
\'etale algebra $k\times k\times k$. Note however that the automorphism group of such a twisted composition is the semi-direct product of a simply connected group of type $\mathrm D_4$
with the symmetric group $S_3$ \cite[36.5]{KMRT}, and therefore is not the required notion for our purposes.  Generalising the notion of twisted compositions to the ring setting
is a possible topic for a future investigation.
 
\end{Rk}

The following lemma is known for classical reduced Albert algebras over fields, and generalises, mutatis mutandis, to the case at hand.

\begin{Lma}\label{Lincl} Let $\calM=(C_1,C_2,C_3,q_1,q_2,q_3,m)$ be a composition of quadratic forms of constant rank 8 over $R$, and let $\Gamma\in (R^*)^3$. The map $\iota:\Aut(\calM)\to\Aut(H(\calM,\Gamma))$ defined, for each $R$-ring $S$ and each $(t_1,t_2,t_3)\in\Aut(\calM)(S)$, by
\[\iota(t_1,t_2,t_3)\left(\sum \alpha_i e_{ii} + \sum u_i[jl]\right)=\sum \alpha_i e_{ii} + \sum t_i(u_i)[jl]\]
is a monomorphism of groups. 
\end{Lma}

\begin{proof} Set $A=H(\calM,\Gamma)$. We first show that $\iota_S(t_1,t_2,t_3)$ is an automorphism of $A_S$ for each $(t_1,t_2,t_3)\in \bAut(\mathcal M)(S)$, $S$ an $R$-ring. 
As it is linear, by \cite[Corollary 18]{P}, it suffices to show that it fixes the unity and preserves 
the cubic norm. The first condition is satisfied by construction; as for the second, recall that
\[N(x)=\alpha_1\alpha_2\alpha_3+\gamma_1\gamma_2\gamma_3\langle m(u_3,u_2),u_1\rangle_1-\sum\gamma a_j\gamma_l\alpha_i q_i(u_i);\]
thus since $t_i$ is an isometry with respect to $q_i$ for each $i$, the norm is preserved by $\iota_S(t_1,t_2,t_3)$ if and only if
\[\langle m(t_3(u_3),t_2(u_2)),t_1(u_1)\rangle_1=\langle m(u_3,u_2),u_1\rangle_1.\]
By definition of $\Aut(\calM)$, the left hand side is equal to $\langle t_1m(u_3,u_2),t_1(u_1)\rangle_1$, which is equal to the right hand side since $t_1$ is an isometry.
The map $\iota$ is further clearly natural, respects composition and is injective, whence the claim follows.
\end{proof}

To understand the quotient $\bAut(A)/\iota(\Aut(\calM))$, we will extend the approach of \cite{ALM}. 

\begin{Def} Let $A$ an Albert algebra over $R$. A \emph{frame for $A$} is a triple $(c_1,c_2,c_3)\in A^3$ such that
\begin{enumerate}
 \item $U_{c_i}1=c_i$, $U_{c_i}c_j=\delta_{ij}$ and $\sum_i c_i=1$, and
 \item the module $U_{c_i}A$ has rank 1.
\end{enumerate}
\end{Def}

The first condition states that $(c_1,c_2,c_3)$ is a complete set of idempotents, and the space $U_{c}A$, for an idempotent $c\in A$, is the \emph{Peirce 2-space} of $c$, which is known to be
a direct summand of $A$ by Peirce decomposition, and hence is projective. The arguments of \cite[4.5]{ALM} (using \cite[6.12]{CTS}) show that the functor $\bF=\bF_A$ from the category of $R$-rings to that of sets, defined by
$\bF(S)$ being the set of all frames for $A_S$, is an affine $R$-scheme. It should be mentioned that there is an extensive theory of frames in Jordan algebras,
in which we shall not delve, the notion introduced above being the particular case we need for our purposes. The interested reader may consult \cite{PR}.

\begin{Ex}\label{Edist} If $A=H(\calM,\Gamma)$, then the triple $\be=(e_1,e_2,e_3)$ is a frame for $A$, which we will call the \emph{distinguished frame} of $A$.
\end{Ex}

\begin{Prp}\label{stabidem} Let $A=H(\calM,\Gamma)$. Under the natural action of $\bAut(A)$ on $\bF_A$ the stabiliser of $(e_1,e_2,e_3)$ is $\iota(\Aut(\calM))$.  
\end{Prp}

\begin{proof} Let $\bH\subseteq\bAut A$ be the stabiliser of $(e_1,e_3,e_3)$. This is a finitely presented group scheme and, by construction, it is clear that 
$\iota(\Aut(\calM))\subseteq \bH$. To check that this inclusion is an equality, it suffices to do so locally with respect to the fppf topology. 
We may thus assume that $\calM=\calM_C$ for a para-octonion algebra $C$. Recall that then, $\Aut(\calM)=\RT(C)$. 
By \cite[4.6]{ALM}, the inclusion $\RT(C)\to\bH$ is an isomorphism on each geometric fibre in the case $\Gamma=\bone$, and in the general case
since over an algebraically closed field, $\Aut(H(C,\Gamma))\simeq\Aut(H(C,\bone))$ via an isomorphism that maps $(e_1,e_2,e_3)$ to itself and commutes with the inclusion of
$\RT(C)$. Since $\RT(C)$ is smooth, hence flat, and of finite presentation, 
we conclude with the fibre-wise isomorphism criterion \cite[$_4$.17.9.5]{EGAIV} and the fact that the property of being an isomorphism over fields
is preserved under descent from an algebraic closure.
\end{proof}

\begin{Prp}\label{Ptorsor} Let $A=H(\mathcal M,\Gamma)$. The map $\Sigma:\Aut(A)\to\bF_A$, defined by $\rho\mapsto (\phi(e_i))_i$ for any $R$-ring $S$ and $\rho\in\Aut(A)(S)$, induces an isomorphism between the fppf quotient
$\bAut(A)/\iota(\Aut(\calM))$ and $\bF_A$. 
\end{Prp}

\begin{proof} With respect to the fppf topology, $\Aut(\calM)$ is locally isomorphic to the smooth group $\RT(C)$ for a para-octonion $R$-algebra $C$. Therefore $\Aut(\calM)$
is smooth, and the quotient is representable by a scheme, which is smooth since so is $\Aut(A)$. By the fibre-wise isomorphism criterion \cite[$_4$17.9.5]{EGAIV} and the fact that the property of being an isomorphism over fields
is preserved under descent from an algebraic closure, it suffices to check that the induced map is
an isomorphism on geometric fibres. We may thus assume that $R$ is an algebraically closed field. But then $\calM\simeq\calM_C$ for some para-octonion algebra $C$, and 
$\Aut(H(\calM,\Gamma))\simeq \Aut(H(C,\Gamma))\simeq\Aut(H(C,\bone))$. It thus suffices to consider the case where the quotient
in question is $\Aut(H(C,\bone))/\RT(C)$, in which we conclude with \cite[4.6]{ALM}.
\end{proof}

Similarly to the previous section, this defines an $\RT(C)$-torsor over $\bF_A$ for the fppf topology, 
and we denote by $\bE^\mathbf{c}$ the fibre of $\mathbf{c}=(c_1,c_2,c_3)\in \bF_A(R)$. The main result of this section is that this torsor is in general non-trivial, even in the split case.

\begin{Thm}\label{TD4} If $C$ is the para-octonion algebra of the (split) complex octonion algebra $O$, then
there exists a smooth $\mathbb C$-ring $S$ and 
a composition $\calM$ of quadratic forms 
over $S$ with $\calM\not\simeq\calM_{C_S}$ and such that $H(\calM,\bone)$ is isomorphic to $H_3(O,\bone)=H(\calM_C,\bone)$.
\end{Thm}

\begin{proof} We have $O=O_0\otimes_{\mathbb R}\mathbb C$, where $O_0$ is the real division octonion algebra. The quadratic form of $O_0$ 
is the Euclidean form $x\mapsto x_1^2+\cdots + x_8^2$. Let $C_0$ be the corresponding real para-octonion algebra, so that $C=C_0\otimes_{\mathbb R}\mathbb C$. It is known (see e.g.\ \cite{AG}) that
$\RT(C_0)\simeq \Spin_{8,\mathbb R}$, which is anisotropic, and from \cite{J1} we know that $\Aut(H_3(O_0,\bone))$ is anisotropic of type $\mathrm F_4$. 
From \cite{Mi} we know that 
$\pi_n(\Spin_8(\mathbb R))$ contains $\pi_n(\Spin_7(\mathbb R))$ as a direct factor, that $\pi_7(\Spin_7(\mathbb R))\simeq \mathbb Z$ and that the seventh homotopy group of the 
compact group of type $\mathrm F_4$ is trivial, and thus does not contain a direct factor isomorphic to $\pi_7(\Spin_{8,\mathbb R}(\mathbb R))$. Now the complex Lie group
$\RT(C)(\mathbb C)\simeq\Spin_{8,\mathbb R}(\mathbb C)$ contains $\Spin_{8,\mathbb R}(\mathbb R)$ as a maximal compact subgroup, and $\Aut(H_3(O,\bone))(\mathbb C)$ contains 
$\Aut(H_3(O_0,\bone))(\mathbb R)$
as a maximal compact subgroup. As in the proof of Theorem \ref{TF4}, it follows by Cartan decomposition that the homotopy groups of these complex Lie groups coincide with those of their
maximal compact subgroups. Since $\mathbf{F}_A$ is affine, we conclude with Lemma \ref{Lhomotopy} and Remark
\ref{Rtrivial}, since the quotient of a smooth group by a
smooth group is smooth.
\end{proof}

\subsection{Compositions of Quadratic Forms and Circle Products}
The above theorem states that $H(\calM,\bone)\simeq H(\calM',\bone)$ may occur when $\calM\not\simeq \calM'$, and Proposition \ref{Ptorsor} implies that those $\calM'$ with 
$H(\calM',\bone)\simeq H(\calM,\bone)$ are obtained as the twists $\bE^\bc\wedge \calM$ as $\bc$ runs through the frames of $H(\calM,\bone)$. In this section, we set out to
describe these twists explicitly. To this end, we shall use the circle product of Jordan algebras, defined by $x\circ y=\{x1y\}$ in terms of the triple product.

\begin{Rk} If $2\in R^*$, then $J$ carries the structure of a linear Jordan algebra with bilinear product $(x,y)\mapsto xy$ defined by $2xy=x\circ y$.
\end{Rk}

Let $\bc=(c_1,c_2,c_3)\in \bF_A(R)$, and for any permutation $(i,j,l)$ of $(1,2,3)$, set
\[C_{l}^\bc=A_1(c_i)\cap A_1(c_j),\] 
where the Peirce 1-space of an idempotent $c$ is defined as
\[A_1(c)=\{x\in A| c\circ x=x\}.\]

\begin{Lma} The restriction of the quadratic trace $S:x\mapsto T(x^\sharp,1)$ to $C_i^\bc$ is a non-degenerate quadratic form of constant rank 8.
\end{Lma}

\begin{proof} The quadratic trace being a quadratic form, so is its restriction to $C_i^\bc$. Since non-degeneracy and rank are invariant under faithfully flat descent, 
it suffices to prove the statement
after a faithfully flat extension $S$ of $R$. By Proposition \ref{Ptorsor}, one can choose $S$ such that $c_i=\phi(e_i)$ for some $\phi\in\bAut(A)(S)$. 
Thus $C_i^\bc\otimes S$ is the image
of ${C_{i}}_S$, which is of rank 8, and the quadratic form in question is the transfer of the restriction of the quadratic trace to ${C_{i}}_S$, which by Proposition \ref{PCOQF} is an invertible scalar multiple of 
the non-degenerate form $q_i$. This proves the statement.
\end{proof}

We write $q_i^\bc$ for the \emph{negative of} the restriction of the quadratic trace to $C_i^\bc$. The choice of sign is made for the next proposition to hold, which describes the
essential ingredient in our twist. 

\begin{Rk} For convenience we introduce some notation. Let $\Gamma=(\gamma_1,\gamma_2,\gamma_3)\in (R^*)^3$. If $\calM=(C_1,C_2,C_3,q_1,q_2,q_3,m)$ is a composition of quadratic forms, then so is 
$(C_1,C_2,C_3,\gamma_2\gamma_3 q_1,\gamma_1\gamma_3 q_2,\gamma_1\gamma_2 q_3,\gamma_1 m)$. We shall denote this composition by $\Gamma\calM$.
\end{Rk}

\begin{Prp}\label{Pdeform} The circle product of $A=H(\calM,\Gamma)$ defines a composition of quadratic forms 
$\calM^\bc=(C_1^\bc,C_2^\bc,C_3^\bc,q_1^\bc,q_2^\bc,q_3^\bc,m_\circ^\bc)$ of rank 8, where $m_\circ^\bc$ is the
restriction of $\circ$ to $C_3^\bc\times C_2^\bc$. If $\bc=(e_1,e_2,e_3)$, then $\calM^\bc$ is isomorphic to $\Gamma\calM$.
In that case the automorphism group of $\calM^\bc$ is $\iota(\Aut(\calM))$.
\end{Prp}

From this it follows that in particular, if $\Gamma=\bone$ and $\bc=(e_1,e_2,e_3)$, then $\calM^\bc$ is isomorphic to $\calM$; if in addition $\calM=\calM_C$
for some symmetric composition algebra $C$, then $\calM^\bc$ has automorphism group $\iota(\RT(C))$.

\begin{proof} If $x\in C_3^\bc$ and $y\in C_2^\bc$, then from \cite[Theorem 31.12(c)]{PR} it follows that $x\circ y\in C_1^\bc$. To show the compatibility of
quadratic forms, we must show that $S(x\circ y)=-S(x)S(y)$. We first assume that $\bc=(e_1,e_2,e_3)$. Writing $x=u_3[12]$ and $y=u_2[31]$ we find, from
\eqref{U} and the fact that $x$ and $y$ are orthogonal to 1 with respect to the bilinear trace, that
\[x\circ y=-(x\times y)\times 1.\]
From the formula for the adjoint in Proposition \ref{PCOQF} we have 
$x\times y=\gamma_1m(u_3,u_2)[23]$,
and using the same formula again we find
\[x\circ y=-(x\times y)\times 1=\gamma_1m(u_3,u_2)[23].\]
By an easy computation, one finds that the quadratic trace satisfies $S(u_i[jl])=-\gamma_j\gamma_lq_i(u_i)$, from which it follows that
\[S(x\circ y)=-\gamma_2\gamma_3q_1(\gamma_1 m(u_3,u_2))=-\gamma_1^2\gamma_2\gamma_3q_3(u_3)q_2(u_2)=-S(x)S(y)\]
as desired. This shows the first statement for this particular choice of $\bc$. For the general case, let $\lambda=S(x\circ y)+S(x)S(y)\in R$. We need to show that $\lambda=0$, for which it suffices
to show that $\lambda\otimes 1=0$ in $R\otimes_R R'=R'$, where $R'$ is a faithfully flat $R$-ring. Now for any $z\in A$, writing $S'$, $T'$ and $\sharp'$ respectively for the quadratic trace, linear
trace, and adjoint of $A_{R'}$, we have $(z\otimes 1)^{\sharp'}=z^{\sharp}\otimes 1$ by \cite[33.1]{PR} and thence by linearity of $T$,
\[S'(z\otimes1)=T'((z\otimes 1)^{\sharp'})=T'(z^{\sharp}\otimes 1)=T(z^{\sharp})\otimes 1=S(z)\otimes 1.\] 
This, together with the bilinearity of the circle product, implies that
\[\lambda\otimes 1=S'((x\otimes 1)\circ(y\otimes 1))+S'(x\otimes1)S'(y\otimes1).\]
By Proposition \ref{Ptorsor} we may choose $R'$ such that
$\Aut(A)(R')$ acts transitively on the set of frames of $A_{R'}$. Thus there is $\phi\in\Aut(A)(R')$ with $\phi(c_i\otimes 1)=e_i\otimes 1$, and since $\phi$ preserves $S'$ and respects the circle
product, we have
\[\lambda\otimes 1=S'(\phi(x\otimes 1)\circ\phi(y\otimes 1))+S'(\phi(x\otimes1))S'(\phi(y\otimes1)).\]
Since $\phi(C_i^\bc\otimes R')=C_i\otimes R'$, we are back to the special case treated above, from which we conclude that $\lambda\otimes 1=0$ as desired. This proves the first statement.

The second statement follows since by the above, $\calM^\bc\simeq\Gamma\calM$ for this choice of $\bc$, and the third statement follows since the quadratic forms and the composition of $\calM^\bc$
are scalar multiples of those of $\calM$.
\end{proof}

We can now give an explicit construction of the twist of a composition of quadratic forms by the above torsor.

\begin{Def} Let $\calM$ be a composition of quadratic forms of rank 8, set $A=H(\calM,\Gamma)$ and let $\bc=(c_1,c_2,c_3)$ be a frame in $A$. The \emph{$(\bc,\Gamma)$-deformation}
 of $\calM$
is the composition $\calM^\bc$ of quadratic forms, defined in Proposition \ref{Pdeform}. 
\end{Def}

Note that $\calM^\bc$ is determined uniquely by the datum $(\calM,\bc,\Gamma)$. We will compare this to the twist $\bE^\bc\wedge(\Gamma\calM)$ of $\Gamma\calM$ by the the torsor 
$\bE^\bc$ defined in the end of Section \ref{Storsor}. 

\begin{Rk} If $\bE$ is any $\Aut(\calM)$-torsor over $\Spec(R)$, then the twist $\bE\wedge\calM$ is defined in analogy to twists of modules and algebras. Concretely, it
is the fppf sheaf associated to the presheaf whose $S$-points, for each $R$-ring $S$, is the composition of quadratic forms
constructed as follows. For each $k\in\{1,2,3\}$, define the equivalence relation $\sim_k$ on $\bE(S)\times {C_k}_S$ by $(\phi,x)\sim_k (\phi',x')$ 
precisely when there exists $(t_1,t_2,t_3)\in\Aut(\calM)$ such that 
\[(\phi,t_k(x))=(\phi'\cdot(t_1,t_2,t_3),x')),\]
where the right action $\cdot$ of $\Aut(\calM)$ on $\bE$ is the one defining the torsor. Denote the equivalence
class of $(\phi,x)$ under $\sim_k$ by $[\phi,x]$. For each $k$, the set $(\bE(S)\times {C_k}_S)/\sim_k$ is a quadratic $S$-module under the scalar multiplication
\[\lambda[\phi,x]=[\phi, \lambda x]\]
and addition
\[[\phi,x]+[\psi,y]=[\phi, x+t_k(y)],\]
where $(t_1,t_2,t_3)$ is uniquely defined by $\psi=\phi\cdot(t_1,t_2,t_3)$, and quadratic form $[\phi,x]\mapsto q_i(x)=S(x)$. Moreover we have a map 
\[(\bE(S)\times {C_3}_S)/\sim_3\times (\bE(S)\times {C_2}_S)/\sim_2\ \longrightarrow\ (\bE(S)\times {C_1}_S)/\sim_1\]
defined by
\[([\phi,x],[\psi,y])\mapsto[\phi,m(x,t_2(y))],\]
where $t_2$ is given by $\psi=\phi\cdot(t_1,t_2,t_3)$. These data define the desired composition of quadratic forms. 
Noting that $\Aut(\calM)=\Aut(\Gamma\calM)$, we can do the same after replacing $\calM$ by $\Gamma\calM$.
\end{Rk}

\begin{Thm} Let $\calM$ be a composition of quadratic forms of rank 8 and let $\bc=(c_1,c_2,c_3)$ be a frame in $S=H(\calM,\Gamma)$.
The twist $\bE^\bc\wedge(\Gamma\calM)$ is canonically isomorphic to $\calM^\bc$.
\end{Thm}

\begin{proof} By Proposition \ref{Pdeform}, we may and will identify $\Gamma\calM$ with $\calM^\be$ for the distinguished frame $\be$ from Example \ref{Edist}, 
upon which $\Aut(\calM)$ acts via the inclusion $\iota$ of Lemma \ref{Lincl}.
Consider the map $\Theta:\bE^\bc\wedge(\Gamma\calM)\to\mathcal{M^\bc}$ induced by the maps
\[\begin{array}{ll}
(\bE(S)\times {C_k}_S)/\sim_k\ \to C_k^\bc(S), & [\phi,x]\mapsto \phi(x),
  \end{array}
\]
which is well-defined by definition of $\sim_k$ and the fact that if $\phi\in\Aut(A)(S)$ maps $\be$ to $\bc$, then it maps $C_k^\be$ to $C_k^\bc$. It is straight-forwardly checked that these maps are a well-defined linear
bijections of modules whenever $\bE^\bc(S)\neq\emptyset$. From Proposition \ref{Pdeform} it follows
that they respect the quadratic forms and are compatible with the composition. We conclude with the definition of the sheaf associated to a presheaf.
\end{proof}

\section{Non-Isomorphic and Non-Isometric Coordinate Algebras}\label{S4}
Over fields, a celebrated theorem that goes back, in its essential form, to Albert and Jacobson \cite{AJ} (see \cite[Theorem 21]{P} for a more general statement) implies that for two 
octonion algebras $C$ and $C'$, we have
\[H_3(C,\bone)\simeq H_3(C',\bone)\Longleftrightarrow q_C\simeq q_{C'}\Longleftrightarrow C\simeq C'.\]
By \cite{G1}, the second equivalence does not hold over arbitrary rings.  In this section we consider the first equivalence and show that this also
fails over rings in general. We begin by showing how a combination of known results implies that the first and third conditions are not equivalent over rings. (For visual clarity, we will, in
this section, often use the notation $\barC$ to denote the para-octonion algebra obtained from the octonion algebra $C$.)

\begin{Prp}\label{Pnontriv} There exists a smooth $\mathbb C$-ring $S$ and two octonion $S$-algebras $C$ and $C'$ such that $C\not\simeq C'$ and $H_3(C,\bone)\simeq H_3(C',\bone)$. 
\end{Prp}

\begin{proof} Set $A=H_3(C,\bone)$. We need to show that the map
\[H^1(S,\Aut(C)){\longrightarrow} H^1(S,\Aut (A)),\]
induced by the subgroup inclusion, has non-trivial kernel for a smooth $\mathbb C$-ring $S$ and octonion $S$-algebra $C$. This inclusion is the composition of the inclusions $i: \RT(\barC)\to \bAut(A)$ of Lemma \ref{Lincl} (noting that
$\RT(\barC)=\bAut(\calM_{\barC})$) and 
$j:\bAut(C)\to\RT(\barC)$ defined by $t\mapsto (t,t,t)$ (see \cite[3.5]{AG}). Thus the map of cohomologies factors as
\[H^1(S,\Aut(C))\overset{j^*}{\longrightarrow}H^1(S,\RT(\barC))\overset{i^*}{\longrightarrow} H^1(S,\bAut A)\]
and has non-trivial kernel since by \cite[4.3]{AG} (which is a variant of \cite[3.5]{G1}), so does $j^*$, for an appropriate (smooth) choice of $S$. This completes the proof.
\end{proof}

From \cite[6.6]{AG} we know that two octonion algebras $C$ and $C'$
have isometric norms if and only if the corresponding compositions $\calM_{\barC}$ and $\calM_{\overline{C'}}$  of quadratic forms are isomorphic. In the previous section we showed that non-isomorphic compositions of quadratic forms
may give rise to isomorphic Albert algebras. To achieve our goal for this section, we need to refine this statement and show that non-isomorphic compositions of quadratic forms 
of the form $\calM_{\barC}$ may give rise to isomorphic Albert algebras. 

Let thus $C$ be an octonion algebra over $R$ and $\barC$ the para-octonion algebra obtained from $C$. Recall that $H_3(C,\bone)=H(\calM_{\barC},\bone)$, and 
consider the inclusions $i: \RT(\barC)=\bAut(\calM_{\barC})\to \bAut(A)$ and 
$j:\bAut(C)\to\RT(\barC)$ from the proof of Proposition \ref{Pnontriv}, and, for every $R$-ring $S$, the induced maps
\[H^1(S,\Aut(C))\overset{j^*}{\longrightarrow}H^1(S,\RT(\barC))\overset{i^*}{\longrightarrow} H^1(S,\bAut A).\]
To show that there exists an octonion $S$-algebra $C'$ with $\calM_{\barC}\not\simeq \calM_{\overline{C'}}$ and 
$A=H_3(C,\bone)\simeq H_3(C',\bone)$,  we need to prove that for some $R$-ring $S$ there is a non-trivial element $\eta\in\Ker(i^*)\cap \mathrm{Im}(j^*)=j^*(\Ker(i^*j^*))$. 

Let $Z=\bAut(A)/\bAut(C)$ and $Y=\bAut(A)/\RT(\barC)$ be the respective fppf quotients with quotient projections $\pi_Z:\Aut(A)\to Z$ and $\pi_Y:\Aut(A)\to Y$.
We will show that for some $z\in Z(R)$, the $\RT(\barC)$-torsor
\[\pi_Z^{-1}(z)\wedge^{\Aut(C)}\RT(\barC)\]
is non-trivial for an appropriate choice of $R$. By Remark \ref{Rquotient}, the class of this torsor is the image under $j^*$ of the class of the torsor $\pi_Z^{-1}(z)$. This implies, together with Remark 
\ref{Rtrivial}, that
for some $R$-ring $S$, this torsor has no $S$-point, and hence there is an octonion $S$-algebra $C'$ with the desired properties. We will therefore prove that the $\RT(\barC)$-torsor
\begin{equation}\label{torsor}
\xymatrix{\Aut(A)\wedge^{\Aut(C)}\RT(\barC)\ar[d]^{\RT(\barC)}\\ Z}
\end{equation}
is non-trivial, i.e.\ does not admit a global section. The following proposition says that this indeed the case in general. We refer the reader to Remark \ref{Rquotient} for the definition of contracted products.

\begin{Thm}\label{TG2} The torsor \eqref{torsor} is non-trivial over $\mathbb R$ for the real division octonion algebra $C$.
\end{Thm}

The proof makes use of the following basic observation.

\begin{Lma} Let $G$ be a group scheme over $R$, $H$ a subgroup scheme of $G$, and $K$ a subgroup scheme of $H$. Then $G\wedge^K H$ and $G\times H/K$ are isomorphic as $H$-torsors over 
$G/K$. Here the right action of $H$ on $G\times H/K$ is given by $(g,x)\cdot h=(gh,h^{-1}x)$ and the structural projection $G\times H/K\to G/K$ is induced by the product map $G\times H\to G, (g,h)\mapsto gh$.
\end{Lma}

\begin{proof} It is straight-forwardly verified that the $H$-action and structural projection of $G\times H/K$ are well-defined; they make $G\times H/K$ an $H$-torsor over $G/K$ since
if over the $R$-ring $S$, the $K$-torsor $G\to G/K$ admits a section $\sigma$, then the structural projection of $G\times H/K$ admits the section $x\mapsto (\sigma(x),e)$, where 
$e$ is
the image of the neutral element of $H$ in $H/K$. To show that the two $H$-torsors in question are isomorphic over $G/K$, consider the map 
\[\Theta:G\wedge^K H\to G\times H/K.\] 
induced by the map $G\times H\to G\times H/K$ defined by $(g,h)\mapsto (gh,\overline{h^{-1}})$, where $\overline{h^{-1}}$ is the quotient image of $h^{-1}$. This map is invariant with respect
to the quotient defining the contracted product. The fact that the induced map is $H$-invariant and lifts the identity of $G/K$ can be straight-forwardly 
verified on the level of presheaves. We conclude with the universal property the sheaf associated to a presheaf and the fact that any morphism of torsors is an isomorphism.
\end{proof}

\begin{proof}[Proof of Theorem \ref{TG2}] The lemma above implies that the two $\RT(\barC)$-torsors 
\[\Aut(A)\wedge^{\Aut(C)}\RT(\barC)\]
and
\[\Aut(A)\times \RT(\barC)/\Aut(C)\]
over $\Aut(A)/\Aut(C)$, which is representable by an affine scheme by \cite[6.12]{CTS}, are isomorphic. By \cite[Theorem 4.1]{AG}, the quotient $\RT(\barC)/\Aut(C)$ is isomorphic to the product of two copies of the Euclidean
7-sphere $\mathbb{S}_7$. Using Lemma \ref{Lhomotopy} and arguing as in the proof of Theorem \ref{TD4}, it thus suffices to show that for some $n$, 
$\pi_n(\Spin_8(\mathbb R))$ is not a direct
factor in 
\[\Gamma=\pi_n(\Aut(A)(\mathbb R))\times\pi_n(\mathbb{S}_7(\mathbb R))^2.\]
Now $\Aut(A)(\mathbb R)$ is compact of type $\mathrm F_4$. It is known that $\pi_{14}(\mathbb{S}_7)\simeq \mathbb{Z}_{120}$, while from \cite{Mi} we get $\pi_{14}(\Aut(A)(\mathbb R))\simeq \mathbb{Z}_2$. Thus in particular
$\Gamma$ has no direct factor isomorphic to $\mathbb{Z}_7$, which, by \cite{Mi} is a direct factor of $\pi_{14}(\Spin_8(\mathbb R))$. This completes the proof.
\end{proof}

\begin{Cor} There exists a smooth $\mathbb{R}$-ring $S$ and octonion $S$-algebras $C$ and $C'$ such that 
$H_3(C,\bone)\simeq H_3(C',\bone)$ and $q_C\not\simeq q_{C'}$.
\end{Cor}

\begin{proof} Let $C_0$ be the real division octonion algebra. Since $\Aut(H_3(C_0,\bone))$ and $\Aut(C_0)$ are smooth, so is their quotient, and it follows from
Theorem \ref{TG2} and Remark \ref{Rtrivial} that there exists a smooth $\mathbb R$-ring $S$ and an octonion $S$-algebra $C'$ with $H_3(C)\simeq H_3(C')$ and $\calM_{\barC}\not
\simeq M_{\overline{C'}}$, where $C=C_0\otimes_{\mathbb R} S$. If $q_C\simeq q_{C'}$,
then the class of $C'$ is in the kernel of the map $H^1(S,\Aut(C))\to H^1(S,\mathbf{O}(q_C))$ induced by the inclusion $\Aut(C)\to\mathbf{O}(q_C)$. But by \cite[Theorem 6.6]{AG}, this kernel
coincides with the kernel of 
\[j^*:H^1(R,\Aut(C))\to H^1(R,\RT(\barC)).\] 
Thus the class of $C'$ is in the kernel of $j^*$, whence $\calM_{\barC}\simeq M_{\overline{C'}}$, a
contradiction. This completes the proof.
 
\end{proof}

\end{document}